\pdfoutput=1
\RequirePackage{ifpdf}
\ifpdf 
\documentclass[pdftex]{sigma}
\else
\documentclass{sigma}
\fi

\usepackage{array,booktabs}
\usepackage{tikz}
\usetikzlibrary{matrix,arrows}

\numberwithin{equation}{section}

\newtheorem{Theorem}{Theorem}[section]
\newtheorem{Corollary}[Theorem]{Corollary}
\newtheorem{Lemma}[Theorem]{Lemma}
\newtheorem{Proposition}[Theorem]{Proposition}
 { \theoremstyle{definition}
\newtheorem{Definition}[Theorem]{Definition}}

\begin{document}
\allowdisplaybreaks

\newcommand{\arXivNumber}{1906.11745}

\renewcommand{\PaperNumber}{075}

\FirstPageHeading

\ShortArticleName{The Racah Algebra as a Subalgebra of the Bannai--Ito Algebra}

\ArticleName{The Racah Algebra as a Subalgebra\\ of the Bannai--Ito Algebra}

\Author{Hau-Wen HUANG}

\AuthorNameForHeading{H.-W.~Huang}

\Address{Department of Mathematics, National Central University, Chung-Li 32001, Taiwan}
\Email{\href{mailto:hauwenh@math.ncu.edu.tw}{hauwenh@math.ncu.edu.tw}}

\ArticleDates{Received May 22, 2020, in final form July 31, 2020; Published online August 10, 2020}

\Abstract{Assume that ${\mathbb F}$ is a field with $\operatorname{char}{\mathbb F}\not=2$. The Racah algebra $\Re$ is a unital associative ${\mathbb F}$-algebra defined by generators and relations. The generators are $A$, $B$, $C$, $D$ and the relations assert that $[A,B]=[B,C]=[C,A]=2D$ and each of $[A,D]+AC-BA$, $[B,D]+BA-CB$, $[C,D]+CB-AC$ is central in $\Re$. The Bannai--Ito algebra $\mathfrak{BI}$ is a~unital associative ${\mathbb F}$-algebra generated by $X$, $Y$, $Z$ and the relations assert that each of $\{X,Y\}-Z$, $\{Y,Z\}-X$, $\{Z,X\}-Y$ is central in $\mathfrak{BI}$. It was discovered that there exists an ${\mathbb F}$-algebra homomorphism $\zeta\colon \Re\to \mathfrak{BI}$ that sends $A \mapsto \frac{(2X-3)(2X+1)}{16}$, $B \mapsto \frac{(2Y-3)(2Y+1)}{16}$, $C \mapsto \frac{(2Z-3)(2Z+1)}{16}$. We show that $\zeta$ is injective and therefore $\Re$ can be considered as an ${\mathbb F}$-subalgebra of $\mathfrak{BI}$. Moreover we show that any Casimir element of $\Re$ can be uniquely expressed as a polynomial in $\{X,Y\}-Z$, $\{Y,Z\}-X$, $\{Z,X\}-Y$ and $X+Y+Z$ with coefficients in ${\mathbb F}$.}


\Classification{81R10; 81R12}

\section{Introduction}

Throughout this paper we adopt the following conventions:
Assume that ${\mathbb F}$ is a field with $\operatorname{char}{\mathbb F}\not=2$. Let ${\mathbb N}$ denote the set of all nonnegative integers. The bracket $[\,,]$ stands for the commutator and the curly bracket $\{\,,\}$ stands for the anticommutator. An algebra is meant to be an associative algebra with unit $1$ and a subalgebra is a subset of the parent algebra which is closed under the operations and has the same unit.

The Racah algebra \cite{zhedanov1988, Levy1965} and the Bannai--Ito algebra \cite{tvz2012} are the ${\mathbb F}$-algebras defined by generators and relations to give the algebraic interpretations of the Racah polynomials and the Bannai--Ito polynomials, respectively. At first, the description of those relations involved several parameters. In recent papers \cite{HR2017,HBI2016,R&BI2015, BI&NW2016} the role of the parameters is replaced by the central elements. The contemporary Racah and Bannai--Ito algebras are defined as follows:
The Racah algebra $\Re$ is an ${\mathbb F}$-algebra generated by $A$, $B$, $C$, $D$ and the relations assert that
\begin{gather*}
[A,B]=[B,C]=[C,A]=2D
\end{gather*}
and each of
\begin{gather*}
\alpha=[A,D]+AC-BA,
\qquad
\beta=[B,D]+BA-CB,
\qquad
\gamma=[C,D]+CB-AC
\end{gather*}
is central in $\Re$. Note that
\begin{gather*}
\delta=A+B+C
\end{gather*}
is also central in $\Re$.
The Bannai--Ito algebra $\mathfrak{BI}$ is an ${\mathbb F}$-algebra generated by $X$, $Y$, $Z$ and the relations assert that each of
\begin{gather*}
\kappa=\{X,Y\}-Z,
\qquad
\lambda=\{Y,Z\}-X,
\qquad
\mu=\{Z,X\}-Y
\end{gather*}
is central in $\mathfrak{BI}$. The applications to the Racah problems for $\mathfrak{su}(2)$, $\mathfrak{su}(1,1)$, $\mathfrak{sl}_{-1}(2)$ and the connections to the Laplace--Dunkl and Dirac--Dunkl equations on the $2$-sphere have been explored in \cite{qBI2019,BI2017, BI2016, BI2014-2,BI2014,gvz2013,R&LD2014,gvz2014,BI2015,zhedanov1988,quadratic1991,Levy1965}. For more information and recent progress, see \cite{Vinet2019,qBI2018, BI2015-2,HR2017,HBI2016,BI2018,SH:2017-1}.

A result of \cite{R&BI2015} made the following link between the Racah algebra $\Re$ and the Bannai--Ito algebra $\mathfrak{BI}$. The standard realization for $\mathfrak{BI}$ is a representation $\pi\colon \mathfrak{BI}\to {\rm End}({\mathbb F}[x])$
given in \cite[Section 4]{tvz2012}.
Inspired by $\pi$, a representation $\tau\colon \Re\to {\rm End}({\mathbb F}[x])$ was constructed in \cite[Section~2]{R&BI2015} as well as an ${\mathbb F}$-algebra homomorphism $\zeta\colon \Re\to \mathfrak{BI}$ that sends
\begin{gather*}
A \mapsto \frac{(2X-3)(2X+1)}{16},
\qquad
B \mapsto \frac{(2Y-3)(2Y+1)}{16},
\qquad
C \mapsto \frac{(2Z-3)(2Z+1)}{16}.
\end{gather*}
Briefly $\tau$ is the composition of $\zeta$ followed by $\pi$. The main result of this paper is to prove that $\zeta$ is injective. To see this we derive the following results.
We show that the monomials
\begin{gather}\label{e:basisR}
A^i B^j C^k D^\ell \alpha^r \beta^s
\qquad
\text{for all $i,j,k,\ell,r,s\in {\mathbb N}$}
\end{gather}
are an ${\mathbb F}$-basis for $\Re$ and the monomials
\begin{gather}\label{eq:basisBI}
X^i Y^j Z^k \kappa^r \lambda^s \mu^t
\qquad \text{for all $i,j,k,r,s,t\in {\mathbb N}$}
\end{gather}
are an ${\mathbb F}$-basis for $\mathfrak{BI}$. We consider the following ${\mathbb F}$-subspaces of $\mathfrak{BI}$ induced from the basis~(\ref{eq:basisBI}) for $\mathfrak{BI}$: Let $w_X,w_Y,w_Z,w_\kappa,w_\lambda,w_\mu\in{\mathbb N}$ be given. For each $n\in {\mathbb N}$ let $\mathfrak{BI}_n$ denote the ${\mathbb F}$-subspace of $\mathfrak{BI}$ spanned by $X^i Y^j Z^k \kappa^r \lambda^s \mu^t$ for all $i,j,k,r,s,t\in {\mathbb N}$ with
\begin{gather*}
w_X i+w_Y j+w_Z k+w_\kappa r+w_\lambda s+w_\mu t\leq n.
\end{gather*}
We show that the sequence $\{\mathfrak{BI}_n\}_{n\in {\mathbb N}}$ is an ${\mathbb N}$-filtration of $\mathfrak{BI}$ if and only if
\begin{gather*}
\max\{w_Z, w_\kappa\}\leq w_X+w_Y,
\qquad
\max\{w_X, w_\lambda\}\leq w_Y+w_Z,
\qquad
\max\{w_Y, w_\mu\} \leq w_Z+w_X.
\end{gather*}
We apply the basis (\ref{e:basisR}) for $\Re$ and the ${\mathbb N}$-filtration $\{\mathfrak{BI}_n\}_{n\in {\mathbb N}}$ of $\mathfrak{BI}$ associated with
\begin{gather*}
(w_X,w_Y,w_Z,w_\kappa ,w_\lambda,w_\mu)=(4,4,6,8,9,9)
\end{gather*}
to conclude the injectivity of $\zeta$.

We regard the Racah algebra $\Re$ as an ${\mathbb F}$-subalgebra of $\mathfrak{BI}$ via $\zeta$.
Let $\mathfrak{C}$ denote the commutative ${\mathbb F}$-subalgebra of $\Re$ generated by $\alpha$, $\beta$, $\gamma$, $\delta$. Extending the setting \cite[Section 2]{gvz2014}, each element of
\begin{gather*}
D^2+A^2+B^2
+\frac{(\delta+2)\{A,B\}-\big\{A^2,B\big\}-\big\{A,B^2\big\}}{2}
+A (\beta-\delta)
+B (\delta-\alpha)+\mathfrak{C}
\end{gather*}
is called a Casimir element of $\Re$ \cite{SH:2017-1}. Each Casimir element of $\Re$ is central in $\Re$. We locate the expressions for the $D_6$-symmetric Casimir elements \cite[Section~5]{SH:2017-1}
of $\Re$ in terms of
\begin{gather*}
\iota=X+Y+Z
\end{gather*}
and $\kappa$, $\lambda$, $\mu$. Note that $\iota$, $\kappa$, $\lambda$, $\mu$ are in the centralizer of $\Re$ in $\mathfrak{BI}$.
Furthermore we apply the ${\mathbb N}$-filtration $\{\mathfrak{BI}_n\}_{n\in {\mathbb N}}$ of $\mathfrak{BI}$ associated with
\begin{gather*}
(w_X,w_Y,w_Z,w_\kappa ,w_\lambda,w_\mu)=(1,1,2,0,0,0)
\end{gather*}
to prove that for any Casimir element $\Omega$ of $\Re$ there exists a unique four-variable polynomial $P(x_1,x_2,x_3,x_4)$ over ${\mathbb F}$ such that
\begin{gather*}
\Omega=P(\iota,\kappa,\lambda,\mu).
\end{gather*}

The outline of this paper is as follows: In Sections~\ref{s:Racah} and~\ref{s:BI} we present the required backgrounds on $\Re$ and $\mathfrak{BI}$, especially the basis (\ref{e:basisR}) for $\Re$ and the criterion for $\{\mathfrak{BI}_n\}_{n\in {\mathbb N}}$ as an ${\mathbb N}$-filtration of $\mathfrak{BI}$.
In Section~\ref{s:zeta} we review the homomorphism $\zeta\colon \Re\to \mathfrak{BI}$ and evaluate the image of $D$ under $\zeta$.
In Section~\ref{s:injective} we give the proof for the injectivity of $\zeta$.
In Section~\ref{s:RCasimir} we show that each Casimir element of $\Re$ can be uniquely expressed as a polynomial in $\iota$, $\kappa$, $\lambda$, $\mu$ over ${\mathbb F}$.

\section[The Racah algebra $\Re$]{The Racah algebra $\boldsymbol{\Re}$}\label{s:Racah}

\begin{Definition}[\cite{HR2017,R&BI2015,zhedanov1988, Levy1965}]\label{defn:R}
The {\it Racah algebra} $\Re$ is an ${\mathbb F}$-algebra defined by generators and relations in the following way. The generators are $A$, $B$, $C$, $D$. The relations assert that
\begin{gather}\label{r:D}
[A,B]=[B,C]=[C,A]=2D
\end{gather}
and each of
\begin{gather*}
[A,D]+AC-BA,
\qquad
[B,D]+BA-CB,
\qquad
[C,D]+CB-AC
\end{gather*}
is central in $\Re$.
\end{Definition}

We define $\alpha$, $\beta$, $\gamma$, $\delta$ as the following elements of $\Re$:
\begin{gather}
\alpha = [A,D]+AC-BA, \label{r:alpha} \\
\beta = [B,D]+BA-CB,\label{r:beta}\\
\gamma = [C,D]+CB-AC,\label{r:gamma}\\
\delta = A+B+C. \label{r:delta}
\end{gather}

\begin{Lemma}[{\cite[Lemma 3.2]{SH:2017-1}}] \label{lem:alp&bet&gam&del}
The following $(i)$--$(iii)$ hold:
\begin{enumerate}\itemsep=0pt
\item[$(i)$] The ${\mathbb F}$-algebra $\Re$ is generated by $A$, $B$, $C$.
\item[$(ii)$] Each of $\alpha$, $\beta$, $\gamma$, $\delta$ is central in $\Re$.
\item[$(iii)$] The sum of $\alpha$, $\beta$, $\gamma$ is equal to zero.
\end{enumerate}
\end{Lemma}

\begin{Proposition}\label{prop:presentation1}
The ${\mathbb F}$-algebra $\Re$ has a presentation with generators $A$, $B$, $C$, $D$, $\alpha$, $\beta$ and relations
\begin{gather}
BA = AB-2D, \label{eq:pres2-1}\\
CB = BC-2D, \label{eq:pres2-2}\\
CA = AC+2D, \label{eq:pres2-3}\\
DA = AD-AB+AC+2D-\alpha, \label{eq:pres2-4}\\
DB = BD-BC+AB-\beta, \label{eq:pres2-5}\\
DC = CD-AC+BC-2D+\alpha+\beta, \label{eq:pres2-6}\\
\alpha A = A\alpha,\qquad \alpha B = B\alpha, \qquad \alpha C = C\alpha, \qquad \alpha D = D\alpha,\label{eq:pres2-7}\\
\beta A = A\beta,\qquad\beta B = B\beta,\qquad\beta C = C\beta,\qquad\beta D = D\beta,\qquad\beta\alpha=\alpha\beta.
\label{eq:pres2-8}
\end{gather}
\end{Proposition}
\begin{proof}
Relations (\ref{eq:pres2-1})--(\ref{eq:pres2-3}) are immediate from (\ref{r:D}).
Relation (\ref{eq:pres2-4}) follows from~(\ref{r:alpha}),~(\ref{eq:pres2-1}). Relation (\ref{eq:pres2-5}) follows from (\ref{r:beta}), (\ref{eq:pres2-1}) and (\ref{eq:pres2-2}). Relation (\ref{eq:pres2-6}) follows from (\ref{r:gamma}), (\ref{eq:pres2-2}) and Lemma \ref{lem:alp&bet&gam&del}(iii). Relations (\ref{eq:pres2-7}) and~(\ref{eq:pres2-8}) follow from Lemma~\ref{lem:alp&bet&gam&del}(ii).
\end{proof}

\begin{Theorem}\label{thm:basisURA3}The elements
\begin{gather}\label{e:basisURA3}
A^i B^j C^k D^\ell \alpha^r \beta^s
\qquad \text{for all $i,j,k,\ell, r,s\in {\mathbb N}$}
\end{gather}
are an ${\mathbb F}$-basis $\Re$.
\end{Theorem}
\begin{proof}To prove the result we invoke the diamond lemma \cite[Theorem 1.2]{diamond1978}.
The relations (\ref{eq:pres2-1})--(\ref{eq:pres2-8}) are regarded as a reduction system.
The ${\mathbb F}$-linear combinations of (\ref{e:basisURA3}) are exactly the irreducible elements under the reduction system. There are no inclusion ambiguities in the reduction system. The nontrivial overlap ambiguities involve the words $CBA$, $DBA$, $DCB$, $DCA$. In any reduction ways, we eventually obtain that
\begin{gather*}
CBA=ABC+2AB-2BC-2AD+2BD-2CD-2\beta,\\
DBA=ABD-2D^2+A^2B-AB^2-2AD+2BD+2AB-2BC-A\beta-B\alpha-2\beta,\\
DCB=BCD-2D^2+B^2C-BC^2-2BD+2CD-2AC+2BC+B\alpha+B\beta-C\beta\\
\hphantom{DCB=}{}-4D+2\alpha+2\beta,\\
DCA=ACD+2D^2-A^2C+AC^2-2AD-2AC+2BC+2CD+A\alpha+A\beta-C\alpha\\
\hphantom{DCA=}{} -4D+2\alpha+2\beta.
\end{gather*}
Hence each of the overlap ambiguities is resolvable.

Let $M$ denote the free monoid with the alphabet set $S=\{A,B,C,D,\alpha,\beta\}$. Let $\ell\colon M\to {\mathbb N}$ denote the length function of~$M$. Consider an element $w=s_1 s_2\cdots s_n\in M$ where \mbox{$s_1,s_2,\ldots,s_n\in S$}. An operation on $w$ is called an {\it elementary operation} if it is one of the following actions on~$w$:
\begin{enumerate}\itemsep=0pt
\item[$\bullet$] We interchange $s_i$ and $s_j$ where $1\leq i<j\leq n$ and the position of $s_j$ is left to the position of $s_i$ in the list
\begin{gather*}
A, \quad B,\quad C,\quad D,\quad \alpha, \quad \beta.
\end{gather*}

\item[$\bullet$] Choose $s_i\in \{B,C,D\}$ and replace $s_i$ by the left neighbor of $s_i$ in the list
\begin{gather*}
A, \quad B, \quad C, \quad D.
\end{gather*}
\end{enumerate}
We define a binary relation $\preceq$ on $M$ as follows: For any $u,w\in M$ we say that
$u\rightarrow w$ whenever $\ell(u)<\ell(w)$ or $u$ is obtained from $w$ by an elementary operation. For any $u,w\in M$ we define $u\preceq w$ if there exist $u_0,u_1,\ldots,u_k\in M$ with $k\in {\mathbb N}$ such that
\begin{gather*}
u=u_0\rightarrow u_1\rightarrow
\cdots
\rightarrow u_{k-1}\rightarrow u_k=w.
\end{gather*}
By construction $\preceq$ is a partial order relation on $M$ satisfying the descending chain condition. Moreover $\preceq$ is a monoid partial order on $M$ compatible with the reduction system (\ref{eq:pres2-1})--(\ref{eq:pres2-8}). Therefore, by diamond lemma the monomials (\ref{e:basisURA3}) form an ${\mathbb F}$-basis for $\Re$.
\end{proof}

Recall that the dihedral group $D_6$ has a presentation with generators $\sigma$, $\tau$ and relations
\begin{gather}\label{D6}
\sigma^2=1,\qquad\tau^6=1,\qquad (\sigma\tau)^2=1.
\end{gather}

\begin{Proposition}[{\cite[Propositions~4.1 and~4.3]{SH:2017-1}}]\label{prop:D6&R} \samepage
There exists a unique $D_6$-action on $\Re$ such that $(i)$, $(ii)$ hold:
\begin{enumerate}\itemsep=0pt
\item[$(i)$] $\sigma$ acts on $\Re$ as an ${\mathbb F}$-algebra antiautomorphism of $\Re$ given in the following way:
\begin{table}[h!]\centering
\begin{tabular}{c|cccc|cccc}
$u$ &$A$ &$B$ &$C$ &$D$
&$\alpha$ &$\beta$ &$\gamma$ &$\delta$
\\
\midrule[1pt]
$\sigma(u)$ &$B$ &$A$ &$C$ &$D$
&$-\beta$ &$-\alpha$ &$-\gamma$ &$\delta$
\end{tabular}
\end{table}

\item[$(ii)$] $\tau$ acts on $\Re$ as an ${\mathbb F}$-algebra antiautomorphism of $\Re$ given in the following way:
\begin{table}[h!]\centering
\begin{tabular}{c|cccc|cccc}
$u$ &$A$ &$B$ &$C$ &$D$
&$\alpha$ &$\beta$ &$\gamma$ &$\delta$
\\
\midrule[1pt]
$\tau(u)$ &$B$ &$C$ &$A$ &$-D$
&$\beta$ &$\gamma$ &$\alpha$ &$\delta$
\end{tabular}
\end{table}
\end{enumerate}
Moreover the $D_6$-action on $\Re$ is faithful.
\end{Proposition}

Let $\mathfrak{C}$ denote the ${\mathbb F}$-subalgebra of $\Re$ generated by $\alpha$, $\beta$, $\gamma$, $\delta$. It follows from Lemma~\ref{lem:alp&bet&gam&del}(ii) that $\mathfrak{C}$ is commutative.

\begin{Definition}[{\cite[Definition 5.2]{SH:2017-1}}]\label{defn:Casimir}
The coset
\begin{gather*}
D^2+A^2+B^2+\frac{(\delta+2)\{A,B\}-\big\{A^2,B\big\}-\big\{A,B^2\big\}}{2} +A (\beta-\delta) +B (\delta-\alpha)+\mathfrak{C}
\end{gather*}
is called the {\it Casimir class} of $\Re$. Each element of the Casimir class of $\Re$ is called a {\it Casimir element} of $\Re$.
\end{Definition}

Define
\begin{gather}
\Omega_A=D^2+\frac{B A C+C A B}{2}+ A^2+B \gamma-C \beta-A \delta,\label{eq:CasA}\\
\Omega_B=D^2+\frac{C B A+A B C}{2}+ B^2+C \alpha-A \gamma-B\delta,\label{eq:CasB}\\
\Omega_C=D^2+\frac{A C B+B C A}{2}+ C^2+A \beta-B\alpha-C\delta.\label{eq:CasC}
\end{gather}
Note that $\Omega_A$, $\Omega_B$, $\Omega_C$ are mutually distinct \cite[Corollary 6.5]{SH:2017-1}.

\begin{Lemma}[{\cite[Proposition 3.7]{SH:2017-1}}]\label{lem:Cas&central}
Each of $\Omega_A$, $\Omega_B$, $\Omega_C$ is a Casimir element.
\end{Lemma}

\begin{Lemma}[{\cite[Lemma 3.6]{SH:2017-1}}]\label{lem:D6&Cas}
The set $\{\Omega_A,\Omega_B,\Omega_C\}$ is invariant under the $D_6$-action on $\Re$.
Moreover the restrictions of $\sigma$ and $\tau$ to $\{\Omega_A,\Omega_B,\Omega_C\}$ are as follows:
\begin{table}[h]
\centering
\begin{tabular}{c|ccc}
$u$ &$\Omega_A$ &$\Omega_B$ &$\Omega_C$\\
\midrule[1pt]
$\sigma(u)$ &$\Omega_B$ &$\Omega_A$ &$\Omega_C$\\
$\tau(u)$ &$\Omega_B$ &$\Omega_C$ &$\Omega_A$
\end{tabular}
\end{table}
\end{Lemma}

\begin{Definition}[{\cite[Section 5]{SH:2017-1}}]
The elements $\Omega_A$, $\Omega_B$, $\Omega_C$ are called the {\it $D_6$-symmetric Casimir elements} of~$\Re$.
\end{Definition}

\section[The Bannai--Ito algebra $\mathfrak{BI}$]{The Bannai--Ito algebra $\boldsymbol{\mathfrak{BI}}$}\label{s:BI}

\begin{Definition}\label{defn:BI}
The {\it Bannai--Ito algebra} $\mathfrak{BI}$ is an ${\mathbb F}$-algebra defined by generators and relations.
The generators are $X$, $Y$, $Z$ and the relations assert that each of
$\{X,Y\}-Z$, $\{Y,Z\}-X$, $\{Z,X\}-Y$ is central in $\mathfrak{BI}$.
\end{Definition}

We define $\iota$, $\kappa$, $\lambda$, $\mu$ as the following elements of $\mathfrak{BI}$:
\begin{gather}
\iota = X+Y+Z, \label{e:iota}\\
\kappa = \{X,Y\}-Z, \label{e:kappa} \\
\lambda = \{Y,Z\}-X, \label{e:lambda}\\
\mu = \{Z,X\}-Y. \label{e:mu}
\end{gather}

\begin{Proposition}\label{prop:D6&BI}
There exists a unique $D_6$-action on $\mathfrak{BI}$ such that $(i)$, $(ii)$ hold:
\begin{enumerate}\itemsep=0pt
\item[$(i)$] $\sigma$ acts on $\mathfrak{BI}$ as an ${\mathbb F}$-algebra antiautomorphism of $\mathfrak{BI}$ given in the following way:
\begin{table}[h!]\centering
\begin{tabular}{c|ccc|cccc}
$u$ &$X$ &$Y$ &$Z$
&$\iota$ &$\kappa$ &$\lambda$ &$\mu$
\\
\midrule[1pt]
$\sigma(u)$ &$Y$ &$X$ &$Z$
&$\iota$ &$\kappa$ &$\mu$ &$\lambda$
\end{tabular}
\end{table}

\item[$(ii)$] $\tau$ acts on $\mathfrak{BI}$ as an ${\mathbb F}$-algebra antiautomorphism of $\mathfrak{BI}$ given in the following way:
\begin{table}[h!]\centering
\begin{tabular}{c|ccc|cccc}
$u$ &$X$ &$Y$ &$Z$
&$\iota$ &$\kappa$ &$\lambda$ &$\mu$\\
\midrule[1pt]
$\tau(u)$ &$Y$ &$Z$ &$X$
&$\iota$ &$\lambda$ &$\mu$ &$\kappa$
\end{tabular}
\end{table}
\end{enumerate}
Moreover the $D_6$-action on $\mathfrak{BI}$ is faithful.
\end{Proposition}

\begin{proof}It is straightforward to verify the existence of the $D_6$-action on $\mathfrak{BI}$ by using (\ref{D6}) and Definition \ref{defn:BI}. Since $D_6$ is generated by $\sigma$ and $\tau$ the uniqueness follows. The ${\mathbb F}$-algebra antiautomorphism of $\mathfrak{BI}$ given in (ii) is of order $6$. It follows from \cite[Lemma~4.2]{SH:2017-1} that the $D_6$-action on $\mathfrak{BI}$ is faithful.
\end{proof}

\begin{Proposition}\label{prop:reductionBI}
The ${\mathbb F}$-algebra $\mathfrak{BI}$ has a presentation with generators $X$, $Y$, $Z$, $\kappa$, $\lambda$, $\mu$ and relations
\begin{gather*}
YX=-XY+Z+\kappa,\qquad
ZY=-YZ+X+\lambda,\qquad
ZX=-XZ+Y+\mu,\\
\kappa X=X\kappa,\qquad\kappa Y=Y\kappa,\qquad\kappa Z=Z\kappa,\\
\lambda X=X\lambda,\qquad \lambda Y=Y\lambda,\qquad \lambda Z=Z\lambda,\qquad \lambda\kappa=\kappa\lambda,\\
\mu X=X\mu,\qquad \mu Y=Y\mu,\qquad\mu Z=Z\mu,\qquad\mu\kappa=\kappa\mu,\qquad\mu\lambda=\lambda\mu.
\end{gather*}
\end{Proposition}
\begin{proof}Immediate from Definition \ref{defn:BI}.
\end{proof}

Applying the diamond lemma to Proposition~\ref{prop:reductionBI}, we obtain the following Poincar\'{e}--Birkhoff--Witt basis for $\mathfrak{BI}$. Since the argument is similar to the proof of Theorem~\ref{thm:basisURA3}, we omit the proof here.

\begin{Theorem}\label{thm:basisBI}
The elements
\begin{gather}\label{e:basisBI}
X^i Y^j Z^k \kappa^r \lambda^s \mu^t
\qquad
\text{for all $i,j,k,r,s,t\in {\mathbb N}$}
\end{gather}
form an ${\mathbb F}$-basis for $\mathfrak{BI}$.
\end{Theorem}

Let $\mathcal A$ denote an ${\mathbb F}$-algebra and let $\mathcal H, \mathcal K$ denote two ${\mathbb F}$-subspaces of $\mathcal A$. The product $\mathcal H\cdot \mathcal K$ is meant to be the ${\mathbb F}$-subspace of $\mathcal A$ spanned by $h\cdot k$ for all $h\in \mathcal H$ and all $k\in \mathcal K$. Recall that an {\it ${\mathbb N}$-filtration} of $\mathcal A$ is a sequence $\{\mathcal A_n\}_{n\in {\mathbb N}}$ of ${\mathbb F}$-subspaces of $\mathcal A$ satisfies the following conditions:
\begin{enumerate}\itemsep=0pt
\item[(N1)] $\bigcup_{n\in {\mathbb N}} \mathcal A_n=\mathcal A$.
\item[(N2)] $\mathcal A_n\subseteq \mathcal A_{n+1}$ for all $n\in {\mathbb N}$.
\item[(N3)] $\mathcal A_m\cdot \mathcal A_n\subseteq \mathcal A_{m+n}$ for all $m,n\in {\mathbb N}$.
\end{enumerate}
For convenience we always let $\mathcal A_{-1}$ denote the zero subspace of $\mathcal A$.

We consider the following ${\mathbb F}$-subspaces of $\mathfrak{BI}$ induced from Theorem \ref{thm:basisBI}: Let $w_X$, $w_Y$, $w_Z$, $w_\kappa$, $w_\lambda$, $w_\mu$ denote the nonnegative integers. For each $n\in {\mathbb N}$ let $\mathfrak{BI}_n$ denote the ${\mathbb F}$-subspace of~$\mathfrak{BI}$ spanned by $X^i Y^j Z^k \kappa^r \lambda^s \mu^t$ for all $i,j,k,r,s,t\in {\mathbb N}$ with
\begin{gather*}
w_Xi+w_Yj+w_Zk+w_\kappa r+w_\lambda s+w_\mu t\leq n.
\end{gather*}
We call $\{\mathfrak{BI}_n\}_{n\in {\mathbb N}}$ the {\it ${\mathbb F}$-subspaces of $\mathfrak{BI}$ associated with $(w_X, w_Y, w_Z, w_\kappa, w_\lambda, w_\mu)$}.
In what follows we give a simple criterion for the above ${\mathbb F}$-subspaces of $\mathfrak{BI}$ to be an ${\mathbb N}$-filtration of $\mathfrak{BI}$.

\begin{Theorem}\label{thm:filtrationBI}
Let $w_X, w_Y, w_Z, w_\kappa, w_\lambda, w_\mu\in {\mathbb N}$.
Let $\{\mathfrak{BI}_n\}_{n\in {\mathbb N}}$ denote the ${\mathbb F}$-subspaces of $\mathfrak{BI}$ associated with $(w_X, w_Y, w_Z, w_\kappa, w_\lambda, w_\mu)$. Then $\{\mathfrak{BI}_n\}_{n\in {\mathbb N}}$ is an ${\mathbb N}$-filtration of $\mathfrak{BI}$
if and only if
\begin{gather}
\max\{w_Z, w_\kappa\} \leq w_X+w_Y,\label{ineq1}\\
\max\{w_X, w_\lambda\} \leq w_Y+w_Z,\label{ineq2}\\
\max\{w_Y, w_\mu\} \leq w_Z+w_X.\label{ineq3}
\end{gather}
\end{Theorem}
\begin{proof} ($\Rightarrow$) By the construction of $\{\mathfrak{BI}_{n}\}_{n\in {\mathbb N}}$ and Theorem \ref{thm:basisBI} the element
\begin{gather*}
Z+\kappa \notin \mathfrak{BI}_{\max\{w_Z, w_\kappa\}-1}.
\end{gather*}
On the other hand, by (N3) we have $\{X,Y\}\in \mathfrak{BI}_{w_X+w_Y}$. The equation (\ref{e:kappa}) implies
\begin{gather*}
Z+\kappa=\{X,Y\}.
\end{gather*}
By the above comments we see that $\mathfrak{BI}_{w_X+w_Y}$ contains $Z+\kappa$ which is not in $\mathfrak{BI}_{\max\{w_Z, w_\kappa\}-1}$. Combined with (N2) the inequality~(\ref{ineq1}) follows. The inequalities (\ref{ineq2}) and (\ref{ineq3}) follow by similar arguments.

($\Leftarrow$) Condition (N1) is immediate from Theorem \ref{thm:basisBI}. Condition (N2) is immediate from the construction of $\{\mathfrak{BI}_n\}_{n\in {\mathbb N}}$. Set $S=\{X,Y,Z,\kappa,\lambda,\mu\}$. For all $n\in {\mathbb N}$, let $I_n$ denote the set of all $(i,j,k,r,s,t)\in {\mathbb N}^6$ with $w_X i+w_Y j+w_Z k+w_\kappa r+w_\lambda s+w_\mu t\leq n$. Let $M$ denote the free monoid with the alphabet set $S$. There exists a unique monoid homomorphism $\tilde w\colon M\to {\mathbb N}$ such that
\begin{gather*}
\tilde w(u)=w_u \qquad \text{for all $u\in S$}.
\end{gather*}
By (\ref{ineq1})--(\ref{ineq3}), for each relation of Proposition \ref{prop:reductionBI}, the value of $\tilde w$ on the monomial in the left-hand side is greater than or equal to those in the right-hand side. Thus, for all $m,n\in {\mathbb N}$ and
for all $(i',j',k',r',s',t')\in I_m$ and $(i'',j'',k'',r'',s'',t'')\in I_n$ the product
\begin{gather*}
X^{i'} Y^{j'} Z^{k'} \kappa^{r'} \lambda^{s'} \mu^{t'}
\cdot
X^{i''} Y^{j''} Z^{k''} \kappa^{r''} \lambda^{s''} \mu^{t''}
\end{gather*}
is equal to an ${\mathbb F}$-linear combination of $X^i Y^j Z^k \kappa^r \lambda^s \mu^t$ for all $(i,j,k,r,s,t)\in I_{m+n}$. In other words (N3) holds. The theorem follows.
\end{proof}

\section[The homomorphism $\zeta\colon \Re\to \mathfrak{BI}$]{The homomorphism $\boldsymbol{\zeta\colon \Re\to \mathfrak{BI}}$}\label{s:zeta}

According to \cite[Section 2]{R&BI2015} there exists an ${\mathbb F}$-algebra homomorphism $\zeta\colon \Re\to \mathfrak{BI}$ and the images of $A$, $B$, $C$, $\alpha$, $\beta$, $\gamma$, $\delta$ under $\zeta$ are as follows:
\begin{Theorem}[\cite{R&BI2015}]
\label{thm:hom}
There exists a unique ${\mathbb F}$-algebra homomorphism $\zeta\colon \Re\to \mathfrak{BI}$ that sends
\allowdisplaybreaks
\begin{gather*}
A \mapsto \frac{(2X-3)(2X+1)}{16},\qquad
B \mapsto \frac{(2 Y-3)(2 Y+1)}{16}, \qquad
C \mapsto \frac{(2 Z-3)(2 Z+1)}{16},
\\
\alpha \mapsto
\frac{(2\iota-\kappa-\mu-3)(\kappa-\mu)}{64},
\qquad
\beta \mapsto
\frac{(2\iota-\lambda-\kappa-3)(\lambda-\kappa)}{64},
\\
\gamma \mapsto
\frac{(2\iota-\mu-\lambda-3)(\mu-\lambda)}{64},
\qquad
\delta \mapsto
\frac{\iota^2-2\iota-\kappa-\lambda-\mu}{4}-\frac{9}{16}.
\end{gather*}
\end{Theorem}

We are now going to evaluate the image of $D$ under $\zeta$.

\begin{Lemma}\label{lem:L}\quad
\begin{enumerate}\itemsep=0pt
\item[$(i)$] The following equations hold in $\mathfrak{BI}$:
\begin{gather*}
\big[X^2,Y\big]=[X,Z],\qquad \big[Y^2,Z\big]=[Y,X],\qquad \big[Z^2,X\big]=[Z,Y],
\\
\big[Y^2,X\big]=[Y,Z], \qquad \big[Z^2,Y\big]=[Z,X],\qquad \big[X^2,Z\big]=[X,Y].
\end{gather*}

\item[$(ii)$] The following elements of $\mathfrak{BI}$ are equal:
\begin{gather*}
\{X,[Z,Y]\},\qquad\{Y,[X,Z]\},\qquad\{Z,[Y,X]\},\\
\big[X^2,Y^2\big],\qquad\big[Y^2,Z^2\big],\qquad\big[Z^2,X^2\big].
\end{gather*}
\end{enumerate}
\end{Lemma}
\begin{proof}
(i) Since $\kappa$ is central in $\mathfrak{BI}$ and by (\ref{e:kappa}) it follows that
\begin{gather*}
[X,Z]=[X,\{X,Y\}].
\end{gather*}
Observe that $[X,\{X,Y\}]=\big[X^2,Y\big]$. Therefore
\begin{gather}\label{e:[X2,Y]}
\big[X^2,Y\big]=[X,Z].
\end{gather}
Applying Proposition \ref{prop:D6&BI} to (\ref{e:[X2,Y]}) yields the remaining equations in (i).

(ii) By Proposition \ref{prop:D6&BI}(ii) it suffices to show that
\begin{gather}
\{X,[Z,Y]\}=\{Y,[X,Z]\},\label{e:L00=L01}\\
\{X,[Z,Y]\}=\big[X^2,Y^2\big].\label{e:L00=L10}
\end{gather}
With trivial cancellations we obtain
\begin{gather}\label{e:[Z,YX]}
\{X,[Z,Y]\}-\{Y,[X,Z]\}=[Z,\{Y,X\}].
\end{gather}
Since $\kappa$ is central in $\mathfrak{BI}$ and by (\ref{e:kappa}) the element $Z$ commutes with $\{Y,X\}$. Hence the right-hand side of (\ref{e:[Z,YX]}) is zero. Therefore (\ref{e:L00=L01}) follows.
Using (\ref{e:kappa}) twice we find that
\begin{gather}\label{e:X2Y2}
X^2 Y^2 =XY^2 X+XZY-XYZ.
\end{gather}
By Proposition \ref{prop:D6&BI}(ii), $\tau^3$ is an ${\mathbb F}$-algebra antiautomorphism of $\mathfrak{BI}$ that fixes $X$, $Y$, $Z$.
Thus, applying $\tau^3$ to (\ref{e:X2Y2}) yields that
\begin{gather}\label{e:Y2X2}
Y^2X^2 =XY^2X+YZX-ZYX.
\end{gather}
Subtracting (\ref{e:Y2X2}) from (\ref{e:X2Y2}) yields (\ref{e:L00=L10}). Hence (ii) follows.
\end{proof}

For convenience we let $L$ denote the common element of $\mathfrak{BI}$ from Lemma \ref{lem:L}(ii).

\begin{Proposition}\label{prop:Dzeta}
The image of $D$ under $\zeta$ is equal to
\begin{gather*}
\frac{[X,Y]+[Y,Z]+[Z,X]+L}{32}.
\end{gather*}
\end{Proposition}
\begin{proof}
By (\ref{r:D}) we have $2D^\zeta=\big[A^\zeta,B^\zeta\big]$.
A direct calculation yields that $[A^\zeta, B^\zeta]$ is equal to
\begin{gather*}
\frac{\big[X^2,Y^2\big]+[X,Y]+\big[Y^2,X\big]+\big[Y,X^2\big]}{16}.
\end{gather*}
By Lemma \ref{lem:L}(i), $\big[Y^2,X\big]=[Y,Z]$ and $\big[Y,X^2\big]=[Z,X]$. By Lemma \ref{lem:L}(ii), $\big[X^2,Y^2\big]=L$. The proposition follows.
\end{proof}

\begin{Corollary}\label{cor:D6&zeta}
For each $g\in D_6$ the following diagram commutes:

\begin{table}[h]
\centering
\begin{tikzpicture}
\matrix(m)[matrix of math nodes,
row sep=2.6em, column sep=2.8em,
text height=1.5ex, text depth=0.25ex]
{
\Re
&\mathfrak{BI}\\
\Re
&\mathfrak{BI}.\\
};
\path[->,font=\scriptsize,>=angle 90]
(m-1-1) edge node[auto] {$\zeta$} (m-1-2)
(m-1-1) edge node[left] {$g$} (m-2-1)
(m-2-1) edge node[auto] {$\zeta$} (m-2-2)
(m-1-2) edge node[auto] {$g$} (m-2-2);
\end{tikzpicture}
\end{table}
\end{Corollary}
\begin{proof}
It is routine to verify the corollary by using Propositions \ref{prop:D6&R}, \ref{prop:D6&BI} and Theorem \ref{thm:hom}.
\end{proof}

We end this section with a comment:
Recall from \cite{BI&NW2016,Huang-BImodules} that a universal analogue of the additive DAHA (double affine Hecke algebra) of type $\big(C_1^\vee,C_1\big)$, denoted by $\mathfrak H$ here, is an ${\mathbb F}$-algebra generated by $t_0$, $t_1$, $t_0^\vee$, $t_1^\vee$ and the relations assert that
\begin{gather*}
t_0+t_1+t_0^\vee+t_1^\vee=-1
\end{gather*}
and each of $t_0^2$, $t_1^2$, $t_0^{\vee 2}$, $t_1^{\vee 2}$ is central in $\mathfrak H$.
By \cite[Proposition~2]{BI&NW2016} there exists an ${\mathbb F}$-algebra isomorphism $\natural\colon \mathfrak{BI}\to \mathfrak H$ that sends
\begin{gather*}
X \mapsto t_0+t_1+\frac{1}{2},
\qquad
Y \mapsto t_0+t_0^\vee+\frac{1}{2},
\qquad
Z \mapsto t_0+t_1^\vee+\frac{1}{2},
\qquad
\iota \mapsto 2t_0+\frac{1}{2},\\
\kappa \mapsto t_0^2-t_1^2-t_0^{\vee 2}+t_1^{\vee 2},
\qquad
\lambda \mapsto t_0^2- t_0^{\vee 2}-t_1^{\vee 2}+t_1^2,
\qquad
\mu \mapsto t_0^2-t_1^{\vee 2}-t_1^2+t_0^{\vee 2}.
\end{gather*}
The universal Askey--Wilson algebra \cite{uaw2011} and the universal DAHA of type $\big(C_1^\vee,C_1\big)$ \cite{DAHA2013} are the $q$-analogues of $\Re$ and $\mathfrak H$, respectively. Therefore
\cite[Theorem~4.1]{DAHA2013} is a $q$-analogue of the homomorphism $\natural\circ \zeta\colon \Re\to \mathfrak H$. Note that $\natural\circ \zeta$ sends
\allowdisplaybreaks
\begin{gather*}
A \mapsto \frac{\big(t_1^\vee+t_0^\vee\big)\big(t_1^\vee+t_0^\vee+2\big)}{4},
\qquad
B \mapsto \frac{\big(t_1+t_1^\vee\big)\big(t_1+t_1^\vee+2\big)}{4},
\\
C \mapsto \frac{\big(t_0^\vee+t_1\big)\big(t_0^\vee+t_1+2\big)}{4},
\\
\alpha \mapsto \frac{\big(t_1^{\vee 2}-t_0^{\vee 2}\big)\big(t_1^2-t_0^2+2t_0-1\big)}{16},
\qquad
\beta \mapsto \frac{\big(t_1^2-t_1^{\vee 2}\big)\big(t_0^{\vee 2}-t_0^2+2t_0-1\big)}{16},
\\
\gamma \mapsto \frac{\big(t_0^{\vee 2}-t_1^2\big)\big(t_1^{\vee 2}-t_0^2+2t_0-1\big)}{16},
\qquad
\delta \mapsto \frac{t_0^2+t_1^2+t_0^{\vee 2}+t_1^{\vee 2}}{4}-\frac{t_0}{2}-\frac{3}{4}.
\end{gather*}

\section[The injectivity of $\zeta$]{The injectivity of $\boldsymbol{\zeta}$}\label{s:injective}

Throughout this section, we let $\{\mathfrak{BI}_n\}_{n\in {\mathbb N}}$ denote the ${\mathbb F}$-subspaces of $\mathfrak{BI}$ associated with
\begin{gather}\label{446899}
(w_X,w_Y,w_Z,w_\kappa,w_\lambda,w_\mu)=(4,4,6,8,9,9).
\end{gather}
Since the number sequence (\ref{446899}) satisfies (\ref{ineq1})--(\ref{ineq3}), it follows from Theorem \ref{thm:filtrationBI} that $\{\mathfrak{BI}_n\}_{n\in {\mathbb N}}$ is an ${\mathbb N}$-filtration of $\mathfrak{BI}$.

\begin{Lemma}\label{lem:BIw1}\quad
\begin{enumerate}\itemsep=0pt
\item[$(i)$] For any even integer $n\geq 0$ the following equations hold:
\begin{gather*}
Y^{n} X =X Y^{n} \pmod{\mathfrak{BI}_{4n+3}},\\
X^{n} Y = Y X^{n} \pmod{\mathfrak{BI}_{4n+3}},\\
Z^{n} Y = Y Z^{n} \pmod{\mathfrak{BI}_{6n+3}},\\
Y^{n} Z = Z Y^{n} \pmod{\mathfrak{BI}_{4n+5}},\\
X^{n} Z = Z X^{n} \pmod{\mathfrak{BI}_{4n+5}},\\
Z^{n} X = X Z^{n} \pmod{\mathfrak{BI}_{6n+3}}.
\end{gather*}

\item[$(ii)$] For any odd integer $n\geq 1$ the following equations hold:
\begin{gather*}
Y^{n} X = -X Y^{n}+\kappa Y^{n-1} \pmod{\mathfrak{BI}_{4n+3}},\\
X^{n} Y = -Y X^{n}+\kappa X^{n-1} \pmod{\mathfrak{BI}_{4n+3}},\\
Y^{n} Z = -Z Y^{n} \pmod{\mathfrak{BI}_{4n+5}},\\
Z^{n} Y = -Y Z^{n}\pmod{\mathfrak{BI}_{6n+3}},\\
X^{n} Z = -Z X^{n} \pmod{\mathfrak{BI}_{4n+5}},\\
Z^{n} X = -X Z^{n} \pmod{\mathfrak{BI}_{6n+3}}.
\end{gather*}
\end{enumerate}
\end{Lemma}
\begin{proof}
All equations are established by routine inductions and using (\ref{e:kappa})--(\ref{e:mu}).
\end{proof}

\begin{Lemma}\label{lem:ABCDnzeta}\quad
\begin{enumerate}\itemsep=0pt
\item[$(i)$] For any integer $n\geq 0$ the following equations hold:
\begin{gather*}
\big(A^\zeta\big)^n = \left(\frac{X}{2}\right)^{2n}\pmod{\mathfrak{BI}_{8n-1}},\\
\big(B^\zeta\big)^n =\left(\frac{Y}{2}\right)^{2n}\pmod{\mathfrak{BI}_{8n-1}},\\
\big(C^\zeta\big)^n =\left(\frac{Z}{2}\right)^{2n}\pmod{\mathfrak{BI}_{12n-1}},\\
\big(\alpha^\zeta\big)^n =\left(\frac{\mu}{8}\right)^{2n} \pmod{\mathfrak{BI}_{18n-1}},\\
\big(\beta^\zeta\big)^n = (-1)^n\left(\frac{\lambda}{8}\right)^{2n} \pmod{\mathfrak{BI}_{18n-1}}.
\end{gather*}

\item[$(ii)$] For any even integer $n\geq 0$ the following equation holds:
\begin{gather*}
\big(D^\zeta\big)^n=\frac{1}{16^n}\sum_{i=0}^{\frac{n}{2}}(-4)^i\binom{\frac{n}{2}}{i}X^{2i} Y^{2i} \kappa^{n-2i}Z^{n}
\pmod{\mathfrak{BI}_{14n-1}}.
\end{gather*}

\item[$(iii)$] For any odd integer $n\geq 1$ the following equation holds:
\begin{gather*}
\big(D^\zeta\big)^n=\frac{1}{16^n}\sum_{i=0}^{\frac{n-1}{2}}(-4)^i
\binom{\frac{n-1}{2}}{i}\big(X^{2i} Y^{2i}\kappa^{n-2i}-2 X^{2i+1} Y^{2i+1}\kappa^{n-2i-1}\big)Z^n \\
\pmod{\mathfrak{BI}_{14n-1}}.
\end{gather*}
\end{enumerate}
\end{Lemma}
\begin{proof}
(i) Immediate from Theorem \ref{thm:hom} and the construction of $\{\mathfrak{BI}_n\}_{n\in {\mathbb N}}$.

(ii) It follows from Proposition \ref{prop:Dzeta} that
\begin{gather*}
D^\zeta=\frac{L}{32} \pmod{\mathfrak{BI}_{13}}.
\end{gather*}
Evaluating $L \bmod{\mathfrak{BI}_{13}}$ by using Lemma \ref{lem:L}(ii) and Lemma \ref{lem:BIw1}(ii) yields that
\begin{gather}\label{e:Dmod13}
D^\zeta=\frac{Z\kappa}{16}-\frac{XYZ}{8} \pmod{\mathfrak{BI}_{13}}.
\end{gather}
Squaring the equation (\ref{e:Dmod13}) a direct calculation shows that
\begin{gather}\label{e:D2mod27}
\big(D^\zeta\big)^2=\frac{Z^2 \kappa^2}{256}-\frac{X^2Y^2Z^2}{64}\pmod{\mathfrak{BI}_{27}}.
\end{gather}
It follows from Lemma \ref{lem:BIw1}(i) that
\begin{gather}\label{e:mod55}
Z^2 \cdot X^2Y^2Z^2=X^2Y^2Z^2\cdot Z^2
\pmod{\mathfrak{BI}_{39}}.
\end{gather}
Now it is routine to derive (ii) by using (\ref{e:D2mod27}) and (\ref{e:mod55}).

(iii) To get (iii), one may multiply (\ref{e:Dmod13}) by the equation from (ii)
and simplify the resulting equation by using Lemma \ref{lem:BIw1}(i).
\end{proof}

\begin{Lemma}\label{lem:ell}
Let $i,j,k,\ell,n,r,s\in {\mathbb N}$ with $8i+8j+12 k+14\ell+18r+18s=n$.
Then the following $(i)$--$(iii)$ hold:
\begin{enumerate}\itemsep=0pt
\item[$(i)$] For all $i',j',k',r',s',t'\in {\mathbb N}$ with $4i'+4j'+6k'+8r'+9s'+9t'=n$ and $r'> \ell$, the coefficient of
\begin{gather*}
X^{i'} Y^{j'} Z^{k'} \kappa^{r'} \lambda^{s'} \mu^{t'}
\end{gather*}
in $\big(A^\zeta\big)^i\big(B^\zeta\big)^j\big(C^\zeta\big)^k\big(D^\zeta\big)^\ell\big(\alpha^\zeta\big)^r\big(\beta^\zeta\big)^s$
with respect to the ${\mathbb F}$-basis {\rm (\ref{e:basisBI})} for $\mathfrak{BI}$ is zero.

\item[$(ii)$] For all $i',j',k',r',s',t'\in {\mathbb N}$ with $4i'+4j'+6k'+8r'+9s'+9t'=n$ and $r'= \ell$, the coefficient of
\begin{gather*}
X^{i'} Y^{j'} Z^{k'} \kappa^{r'} \lambda^{s'} \mu^{t'}
\end{gather*}
in $\big(A^\zeta\big)^i\big(B^\zeta\big)^j\big(C^\zeta\big)^k\big(D^\zeta\big)^\ell\big(\alpha^\zeta\big)^r\big(\beta^\zeta\big)^s$
with respect to the ${\mathbb F}$-basis {\rm (\ref{e:basisBI})} for $\mathfrak{BI}$ is nonzero if and only if
\begin{gather*}
(i',j',k',s',t')=(2i, 2j, 2k+\ell, 2s, 2r).
\end{gather*}

\item[$(iii)$] The coefficient of
\begin{gather*}
X^{2i} Y^{2j} Z^{2k+\ell} \kappa^\ell \lambda^{2s} \mu^{2r}
\end{gather*}
in
$\big(A^\zeta\big)^i\big(B^\zeta\big)^j\big(C^\zeta\big)^k\big(D^\zeta\big)^\ell\big(\alpha^\zeta\big)^r\big(\beta^\zeta\big)^s$
with respect to the ${\mathbb F}$-basis {\rm (\ref{e:basisBI})} for $\mathfrak{BI}$ is
\begin{gather*}
(-1)^s 4^{-i-j-k-2\ell-3r-3s}.
\end{gather*}
\end{enumerate}
\end{Lemma}
\begin{proof}Using Lemmas \ref{lem:BIw1}(i) and~\ref{lem:ABCDnzeta} one may express
\begin{gather*}
\big(A^\zeta\big)^i
\big(B^\zeta\big)^j
\big(C^\zeta\big)^k
\big(D^\zeta\big)^\ell
\big(\alpha^\zeta\big)^r
\big(\beta^\zeta\big)^s
+\mathfrak{BI}_{n-1}
\end{gather*}
as an ${\mathbb F}$-linear combination of $X^{i'} Y^{j'} Z^{k'} \kappa^{r'} \lambda^{s'} \mu^{t'}+\mathfrak{BI}_{n-1}$ for all $i',j',k',r',s',t'\in {\mathbb N}$ with $4i'+4j'+6k'+8r'+9s'+9t'=n$. The lemma follows from the expression.
\end{proof}

\begin{Theorem}\label{thm:injective}
The homomorphism $\zeta\colon \Re\to \mathfrak{BI}$ is injective.
\end{Theorem}
\begin{proof}
Suppose on the contrary that there exists a nonzero element $I$ in the kernel of $\zeta$. For all $i,j,k,\ell,r,s\in {\mathbb N}$ let $c(i,j,k,\ell,r,s)$ denote the coefficient of
\begin{gather*}
A^iB^jC^kD^\ell\alpha^r\beta^s
\end{gather*}
in $I$ with respect to the ${\mathbb F}$-basis (\ref{e:basisURA3}) for $\Re$. Let $S$ denote the set of all $(i,j,k,\ell,r,s)\in {\mathbb N}^6$ with $c(i,j,k,\ell,r,s)\not=0$. For each $n\in {\mathbb N}$ we let $S(n)$ denote the set of all $(i,j,k,\ell,r,s)\in S$ with
$8i+8j+12 k+14\ell+18r+18s=n$. We may write
\begin{gather}\label{e:I}
I=\sum_{n\in {\mathbb N}} \sum_{(i,j,k,\ell,r,s)\in S(n)} c(i,j,k,\ell,r,s) A^i B^j C^k D^\ell\alpha^r\beta^s.
\end{gather}

Applying $\zeta$ to (\ref{e:I}) we have
\begin{gather}\label{e:Izeta}
0=\sum_{n\in {\mathbb N}}
\sum_{(i,j,k,\ell,r,s)\in S(n)}c(i,j,k,\ell,r,s)\big(A^\zeta\big)^i \big(B^\zeta\big)^j \big(C^\zeta\big)^k\big(D^\zeta\big)^\ell \big(\alpha^\zeta\big)^r \big(\beta^\zeta\big)^s.
\end{gather}
Since $I\not=0$ there exists at least one $n\in {\mathbb N}$ with $S(n)\not=\varnothing$.
Set
\begin{gather*}
N=\max\{n\,|\, S(n)\not=\varnothing\}.
\end{gather*}
Among the elements in $S(N)$ we choose a $6$-tuple $(i,j,k,\ell,r,s)$ that has the maximum value at~$\ell$. In what follows we evaluate the coefficient of
\begin{gather}\label{e:topnomial}
X^{2i} Y^{2j} Z^{2k+\ell} \kappa^{\ell} \lambda^{2s} \mu^{2r}
\end{gather}
in the right-hand side of (\ref{e:Izeta}) with respect to the ${\mathbb F}$-basis (\ref{e:basisBI}) for $\mathfrak{BI}$. Denote by $c$ the coefficient. Suppose that $(i',j',k',\ell',r',s')$ is a $6$-tuple in $S(n)$ for some $n\in {\mathbb N}$ such that
\begin{gather}\label{e:another}
\big(A^\zeta\big)^{i'} \big(B^\zeta\big)^{j'} \big(C^\zeta\big)^{k'} \big(D^\zeta\big)^{\ell'} \big(\alpha^\zeta\big)^{r'} \big(\beta^\zeta\big)^{s'}
\end{gather}
contributes to the coefficient~$c$. By Theorem \ref{thm:basisBI} the monomial (\ref{e:topnomial}) lies in $\mathfrak{BI}_N$ not in $\mathfrak{BI}_{N-1}$. By Lemma \ref{lem:ABCDnzeta} the term (\ref{e:another}) lies in $\mathfrak{BI}_n$. It follows from (N2) that $n\geq N$ and the maximality of $N$ implies $n=N$.
By Lemma~\ref{lem:ell}(i) we have $\ell'\geq \ell$ and the maximality of $\ell$ forces that $\ell'=\ell$. Combined with Lemma~\ref{lem:ell}(ii) this yields that $(i',j',k',r',s')=(i,j,k,r,s)$. Therefore
\begin{gather*}
 \big(A^\zeta\big)^i \big(B^\zeta\big)^j \big(C^\zeta\big)^k \big(D^\zeta\big)^\ell \big(\alpha^\zeta\big)^r \big(\beta^\zeta\big)^s
\end{gather*}
is the only summand in the right-hand side of (\ref{e:Izeta}) contributes to the coefficient~$c$. By Lem\-ma~\ref{lem:ell}(iii) the coefficient $c$ is the nonzero scalar \begin{gather*}
(-1)^s\cdot 4^{-i-j-k-2\ell-3r-3s}\cdot c(i,j,k,\ell,r,s).
\end{gather*}
It follows from Theorem \ref{thm:basisBI} that the right-hand side of (\ref{e:Izeta}) is nonzero, a contradiction. The theorem follows.
\end{proof}

As a consequence of Theorem \ref{thm:injective} the ${\mathbb F}$-algebra homomorphism $\natural\circ \zeta\colon \Re\to \mathfrak H$ described in Section~\ref{s:zeta} is injective. Note that \cite[Theorem~4.5]{DAHA2013} is a $q$-analogue of the injectivity for $\natural\circ \zeta$.

\section[The images of the Casimir elments of $\Re$ under $\zeta$]{The images of the Casimir elments of $\boldsymbol{\Re}$ under $\boldsymbol{\zeta}$}\label{s:RCasimir}

In light of Theorem \ref{thm:injective} the Racah algebra $\Re$ can be viewed as an ${\mathbb F}$-subalgebra of the Bannai--Ito algebra $\mathfrak{BI}$ via $\zeta$.

\begin{Lemma}\label{lem:iota}The element $\iota$ is in the centralizer of $\Re$ in $\mathfrak{BI}$.
\end{Lemma}
\begin{proof}By Theorem \ref{thm:hom} and (\ref{e:iota}) the commutator $[\iota,A]$ is equal to $\frac{1}{4}$ times
\begin{gather}\label{e:[iota,A]}
\big[Y+Z,X^2\big]-[Y+Z,X].
\end{gather}
Simplifying (\ref{e:[iota,A]}) by using Lemma \ref{lem:L}(i) yields that (\ref{e:[iota,A]}) is zero. Therefore $\iota$ commutes with $A$. Similarly $\iota$ commutes with $B$ and $C$. Combined with Lemma \ref{lem:alp&bet&gam&del}(i) the lemma follows.
\end{proof}

By Lemma \ref{lem:iota} each of $\iota$, $\kappa$, $\lambda$, $\mu$ lies in the centralizer of $\Re$ in $\mathfrak{BI}$. The intention of the final section is to show that each Casimir element of $\Re$ can be uniquely expressed as a polynomial in~$\iota$, $\kappa$, $\lambda$, $\mu$ with coefficients in~${\mathbb F}$.

Throughout this section, let $\{\mathfrak{BI}_n\}_{n\in {\mathbb N}}$ denote the ${\mathbb F}$-subspaces of $\mathfrak{BI}$ associated with
\begin{gather}\label{112000}
(w_X,w_Y,w_Z,w_\kappa,w_\lambda,w_\mu)=(1,1,2,0,0,0).
\end{gather}
Since the sequence (\ref{112000}) satisfies (\ref{ineq1})--(\ref{ineq3}), it follows from Theorem \ref{thm:filtrationBI} that $\{\mathfrak{BI}_n\}_{n\in {\mathbb N}}$ is an ${\mathbb N}$-filtration of $\mathfrak{BI}$.

\begin{Lemma}\label{lem:iota=Z}
$Z^n=\iota^n \bmod{\mathfrak{BI}_{2n-1}}$ for all $n\in {\mathbb N}$.
\end{Lemma}
\begin{proof}
Proceed by induction on $n$. It is trivial for $n=0$. By (\ref{e:iota}) we have
\begin{gather}\label{e:n=1}
\iota-Z=X+Y\in \mathfrak{BI}_1.
\end{gather}
Hence the lemma holds for $n=1$.
 Suppose that $n\geq 2$. We divide $\iota^n-Z^n$ into
\begin{gather}\label{iotan-Zn}
Z\big(\iota^{n-1}-Z^{n-1}\big)+(\iota-Z)\iota^{n-1}.
\end{gather}
Since $Z\in \mathfrak{BI}_2$ and by induction hypothesis, the first summand of (\ref{iotan-Zn}) is in $\mathfrak{BI}_{2n-1}$. By (\ref{e:iota}) the element $\iota\in \mathfrak{BI}_2$ and hence $\iota^{n-1}\in \mathfrak{BI}_{2n-2}$. Combined with~(\ref{e:n=1}) the second summand of~(\ref{iotan-Zn}) is in~$\mathfrak{BI}_{2n-1}$. The lemma follows.
\end{proof}

\begin{Lemma}\label{lem:basis1_cosetBI}\quad
\begin{enumerate}\itemsep=0pt
\item[$(i)$] For all $n\in {\mathbb N}$ the elements
\begin{gather*}
X^i Y^j Z^k \kappa^r \lambda^s \mu^t+\mathfrak{BI}_{n-1}
\qquad
\text{for all $i,j,k,r,s,t\in {\mathbb N}$ with $i+j+2k=n$}
\end{gather*}
are an ${\mathbb F}$-basis for $\mathfrak{BI}_n/\mathfrak{BI}_{n-1}$.

\item[$(ii)$] For all $n\in {\mathbb N}$ the elements
\begin{gather*}
X^i Y^j \iota^k \kappa^r \lambda^s \mu^t+\mathfrak{BI}_{n-1}
\qquad
\text{for all $i,j,k,r,s,t\in {\mathbb N}$ with $i+j+2k=n$}
\end{gather*}
are an ${\mathbb F}$-basis for $\mathfrak{BI}_n/\mathfrak{BI}_{n-1}$.

\item[$(iii)$] For all $n\in {\mathbb N}$ the elements
\begin{gather*}
X^i Y^j \iota^k \kappa^r \lambda^s \mu^t
\qquad
\text{for all $i,j,k,r,s,t\in {\mathbb N}$ with $i+j+2k\leq n$}
\end{gather*}
are an ${\mathbb F}$-basis for $\mathfrak{BI}_n$.
\end{enumerate}
\end{Lemma}
\begin{proof}
(i) Immediate from Theorem \ref{thm:basisBI} and the construction of $\{\mathfrak{BI}_n\}_{n\in {\mathbb N}}$.

(ii) Immediate from Lemma \ref{lem:iota=Z} and (i).

(iii) Using (ii) the statement (iii) follows by a routine induction on $n$.
\end{proof}

\begin{Theorem}\label{thm:basisBI2}
The elements
\begin{gather}\label{e:basisBI2}
X^i Y^j \iota^k \kappa^r \lambda^s \mu^t
\qquad
\text{for all $i,j,k,r,s,t\in {\mathbb N}$}
\end{gather}
are an ${\mathbb F}$-basis for $\mathfrak{BI}$.
\end{Theorem}
\begin{proof}
Immediate from (N1) and Lemma \ref{lem:basis1_cosetBI}(iii).
\end{proof}

\begin{Corollary}\label{cor:ind_iklm}
The elements $\iota$, $\kappa$, $\lambda$, $\mu$ of $\mathfrak{BI}$ are algebraically independent over ${\mathbb F}$.
\end{Corollary}
\begin{proof}Immediate from Theorem \ref{thm:basisBI2}.
\end{proof}

\begin{Lemma}\label{lem:presentationBI2}
The ${\mathbb F}$-algebra $\mathfrak{BI}$ has a presentation with generators $X$, $Y$, $\iota$, $\kappa$, $\lambda$, $\mu$ and relations
\begin{gather*}
YX=-XY-X-Y+\iota+\kappa,\\
\iota Y =2 Y^2 -Y\iota- Y+\iota+\kappa+\lambda,\\
\iota X =2 X^2-X\iota- X+\iota+\kappa+\mu,\\
\kappa X=X\kappa,\qquad\kappa Y=Y\kappa,\qquad\kappa \iota=\iota \kappa,\\
\lambda X=X\lambda,\qquad\lambda Y=Y\lambda,\qquad\lambda \iota=\iota\lambda,\qquad\lambda\kappa=\kappa\lambda,\\
\mu X=X\mu,\qquad\mu Y=Y\mu,\qquad\mu \iota=\iota\mu,\qquad\mu\kappa=\kappa\mu,\qquad\mu\lambda=\lambda\mu.
\end{gather*}
\end{Lemma}
\begin{proof}
This is a reformulation of Proposition \ref{prop:reductionBI} by using (\ref{e:iota}).
\end{proof}

Recall the $D_6$-symmetric Casimir elements $\Omega_A$, $\Omega_B$, $\Omega_C$ of $\Re$ from (\ref{eq:CasA})--(\ref{eq:CasC}).

\begin{Proposition}\label{prop:D6Casmir}
The $D_6$-symmetric Casimir elements $\Omega_A$, $\Omega_B$, $\Omega_C$ of $\Re$ have the following expressions:
\begin{gather}
\Omega_A=\frac{\lambda(\lambda-2\iota+3)\big(4\iota^2-8\iota-4\kappa-4\lambda-4\mu+7\big)}{1024}-\Gamma,\label{OmegaA}\\
\Omega_B=\frac{\mu(\mu-2\iota+3)\big(4\iota^2-8\iota-4\kappa-4\lambda-4\mu+7\big)}{1024}-\Gamma,\label{OmegaB}\\
\Omega_C=\frac{\kappa(\kappa-2\iota+3)\big(4\iota^2-8\iota-4\kappa-4\lambda-4\mu+7\big)}{1024}-\Gamma,\label{OmegaC}
\end{gather}
where
\begin{gather*}
\Gamma=\frac{3(2\iota+3)(2\iota+1)(2\iota-5)(2\iota-7)}{4096}-\frac{(2\iota+1)(6\iota-13)(\kappa+\lambda+\mu)}{512}\\
\hphantom{\Gamma=}{} + \frac{(\kappa+\lambda+\mu)(\kappa+\lambda+\mu+4)}{64}
-\frac{(2\iota-3)(\kappa\lambda+\lambda\mu+\mu\kappa)}{512}
+\frac{\kappa\lambda\mu}{256}.
\end{gather*}
\end{Proposition}
\begin{proof}Applying Theorem \ref{thm:hom} and Proposition \ref{prop:Dzeta} to (\ref{eq:CasA}) and replacing $Z$ by $\iota-X-Y$, we may express $\Omega_A$ in terms of $X$, $Y$, $\iota$, $\kappa$, $\lambda$, $\mu$. To get~(\ref{OmegaA}) we apply Lemma \ref{lem:presentationBI2} to express the resulting expression as an ${\mathbb F}$-linear combination of~(\ref{e:basisBI2}). Combined with Lemma \ref{lem:D6&Cas} and Proposition~\ref{prop:D6&BI} we obtain~(\ref{OmegaB}) and~(\ref{OmegaC}).
\end{proof}

\begin{Theorem}\label{thm:Casimir}For each Casimir element $\Omega$ of $\Re$ there exists a unique four-variable polynomial $P(x_1,x_2,x_3,x_4)$ over ${\mathbb F}$ such that
\begin{gather*}
\Omega=P(\iota,\kappa,\lambda,\mu).
\end{gather*}
\end{Theorem}
\begin{proof}
By Definition \ref{defn:Casimir} there exists a four-variable polynomial $Q(y_1,y_2,y_3,y_4)$ over ${\mathbb F}$ such that
$\Omega=\Omega_A+Q(\alpha,\beta,\gamma,\delta)$.
Set
$\widehat Q(x_1,x_2,x_3,x_4) = Q(y_1,y_2,y_3,y_4)$
by substituting
\begin{gather*}
y_1 =\frac{(2 x_1- x_2 -x_4-3)(x_2-x_4)}{64},\qquad y_2 =\frac{(2 x_1-x_3-x_2-3)(x_3-x_2)}{64},\\
y_3 =\frac{(2 x_1-x_4-x_3-3)(x_4-x_3)}{64},\qquad y_4 =\frac{x_1^2-2x_1-x_2-x_3-x_4}{4}-\frac{9}{16}.
\end{gather*}
It follows from Theorem \ref{thm:hom} that $\Omega=\Omega_A+\widehat Q(\iota,\kappa,\lambda,\mu)$. Combined with Proposition~\ref{prop:D6Casmir} the existence follows.
The uniqueness is immediate from Corollary~\ref{cor:ind_iklm}.
\end{proof}

\subsection*{Acknowledgements}

The research is supported by the Ministry of Science and Technology of Taiwan under the project MOST 106-2628-M-008-001-MY4.

\pdfbookmark[1]{References}{ref}
\LastPageEnding

\end{document}